\documentclass[12pt]{article}

\usepackage{amssymb}
\usepackage[cal=boondox]{mathalfa}
\usepackage{amsthm}
\usepackage{amsmath}
\usepackage{graphicx}
\usepackage{fullpage}
\usepackage{color}
\usepackage{enumerate}
 \numberwithin{equation}{section}
\usepackage{booktabs}
\allowdisplaybreaks

\usepackage[pdftex]{hyperref}
 \usepackage{mathtools}
 \mathtoolsset{showonlyrefs}

\theoremstyle{plain}
\newtheorem{thm}{Theorem}[section]
\newtheorem{cor}[thm]{Corollary}

\newtheorem{lem}[thm]{Lemma}
\newtheorem{prop}[thm]{Proposition}

\theoremstyle{definition}
\newtheorem{defn}[thm]{Definition}

\theoremstyle{remark}

\newtheorem{rem}[thm]{Remark}

\newcommand{\N}{\mathbb{N}}
\newcommand{\R}{\mathbb{R}}

\newcommand{\Z}{\mathbb{Z}} 
\newcommand{\E}{\mathbb{E}}

\newcommand{\bp}{\begin{proof}[\ensuremath{\mathbf{Proof}}]}
\newcommand{\bs}{\begin{proof}[\ensuremath{\mathbf{Solution}}]}
\newcommand{\ep}{\end{proof}}
\newcommand{\be}{\begin{equation}}
\newcommand{\ee}{\end{equation}}
\newcommand{\id}{\text{id}_{\R}}

\begin{document}

% Title
\title{Lagrangian coordinates for the sticky particle system}

% Name
\author{Ryan Hynd}

\maketitle

%  Abstract  
\begin{abstract}
The sticky particle system is a system of partial differential equations which assert 
the conservation of mass and momentum of a collection of particles that interact only via inelastic collisions. These equations arise in Zel'dovich's theory for the formation of large scale structures in the universe. We will show that this system of equations has a solution in one spatial dimension for given initial conditions by generating a trajectory mapping in Lagrangian coordinates.
\end{abstract}

%%%%%%%%%%%%%%%%%%%%%%%%%%%%%%%%%%%%%%%%%%%%%%%%
\section{Introduction}
% Sticky particle system 
In this paper, we will study the {\it sticky particle system} (SPS) in one spatial dimension
\be\label{SPS}
\begin{cases}
\hspace{.27in}\partial_t\rho +\partial_x(\rho v)=0\\
\partial_t(\rho v) +\partial_x(\rho v^2)=0.
\end{cases}
\ee
These equations hold in $\R\times (0,\infty)$ and are typically supplemented with given initial conditions 
\be\label{Init}
\rho|_{t=0}=\rho_0\quad \text{and}\quad v|_{t=0}=v_0.
\ee
The first equation listed in \eqref{SPS} expresses the conservation of mass and the second expresses the conservation of momentum. The unknowns are a pair $\rho$ and $v$ which represent the respective mass density and velocity of a collection of particles that move along the real line and interact via inelastic collisions. Likewise, $\rho_0$ is the associated initial mass distribution and $v_0$ is the corresponding initial velocity. 

\begin{figure}[h]
\centering
 \includegraphics[width=.7\textwidth]{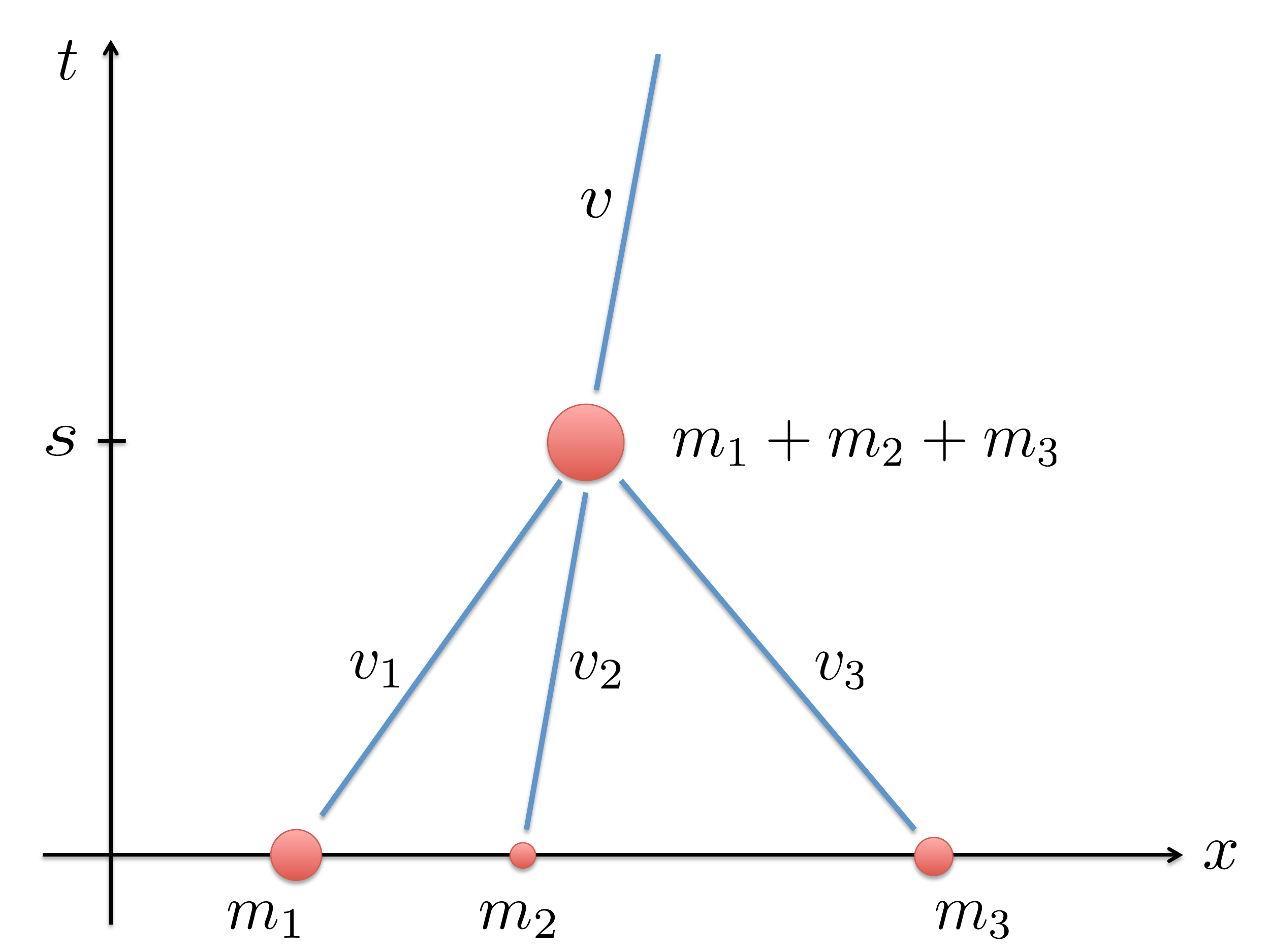}
 \caption{Three point masses $m_1, m_2, m_3$ undergo a perfectly inelastic collision at time $s$.  Here $v$ satisfies 
 $m_1v_1+m_2v_2+m_3v_3=(m_1+m_2+m_3)v$. Also note that these masses are displayed larger than points to emphasize that they are possibly distinct. }\label{Fig1}
\end{figure}

% Physics/inelastic collisions 
\par The SPS first arose in cosmology in the study of galaxy formation. In particular, Zel'dovich considered these equations in three spatial dimensions when he studied the evolution of matter at low temperatures that wasn't subject to pressure \cite{Gurbatov, Zeldovich}.  To get an idea for the physics involved, we will study a simple scenario in which finitely many particles are constrained to move on the real line.  We assume that these particles move in straight line trajectories when they are not in contact; however, particles undergo perfectly inelastic collisions once they collide.  For example, if the particles with masses $m_1,\dots, m_k>0$ have respective velocities $v_1,\dots, v_k$ before a collision, they will join to form a single particle of mass $m_1+\dots +m_k$ upon collision which moves with velocity $v$ chosen to satisfy 
$$
m_1v_1+\dots +m_kv_k=(m_1+\dots +m_k)v.
$$ 
See Figure \ref{Fig1} for an example.

\par For each $i\in \{1,\dots, N\}$ and $t\ge 0$, we write $\gamma_i(t)\in \R$ for the position of mass $m_i$ at time $t$, which could be by itself or part of a larger mass if it has already collided with another particle. This specification allows us to associate trajectories $\gamma_1,\dots, \gamma_N:[0,\infty)\rightarrow\R$ that track the positions of the respective point masses $m_1,\dots, m_N$.  See Figure \ref{Fig2} for a schematic diagram. It turns out that these trajectories have various natural properties including
$$
\gamma_i(t)=\gamma_j(t), \quad t\ge s
$$
whenever $\gamma_i(s)=\gamma_j(s)$. 

\begin{figure}[h]
\centering
 \includegraphics[width=.7\textwidth]{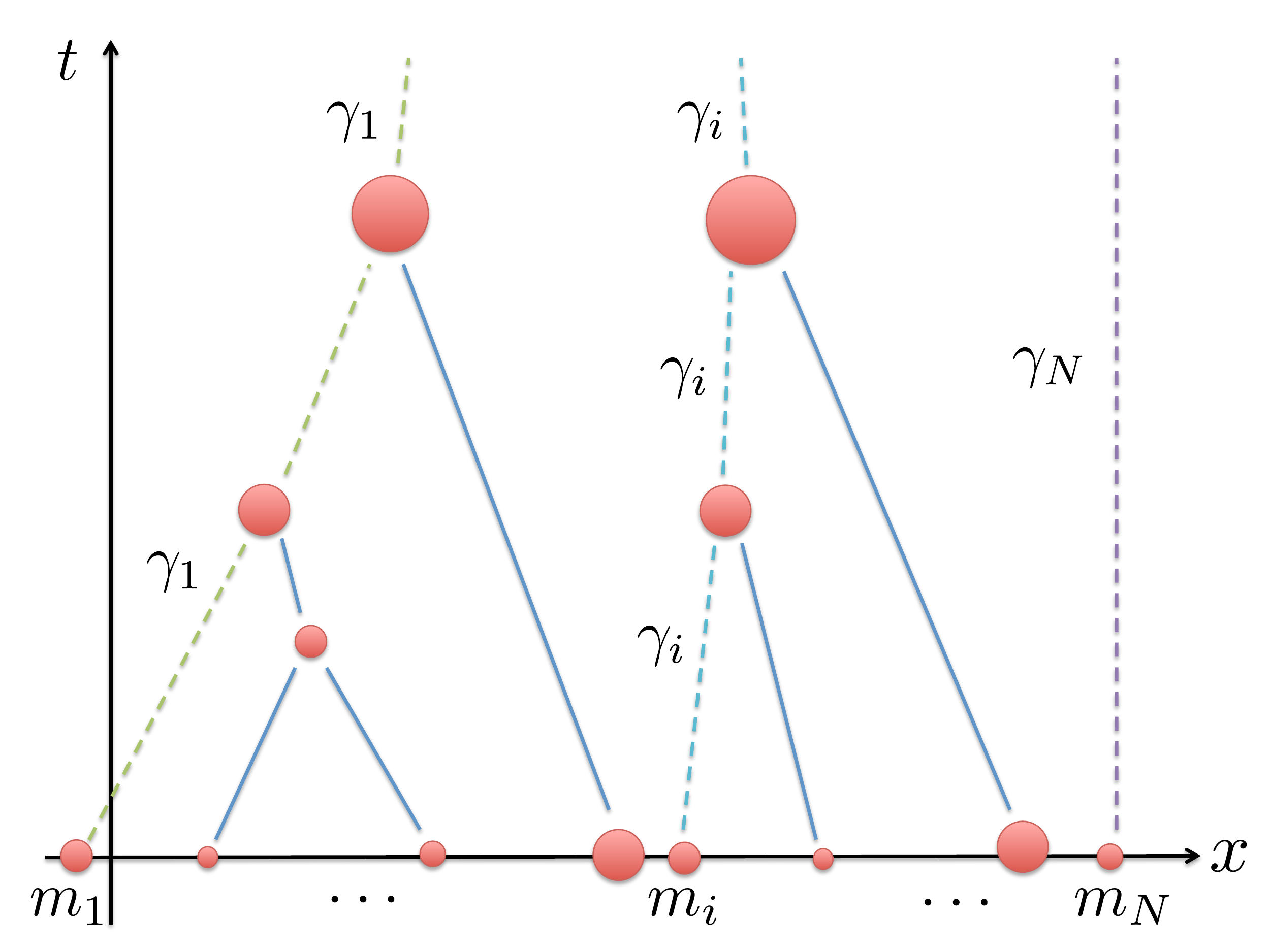}
 \caption{Sticky particle trajectories $\gamma_1,\dots, \gamma_N:[0,\infty)\rightarrow \R$ that track the positions of the respective masses $m_1,\dots, m_N$. Three trajectories $\gamma_1,\gamma_i$ and $\gamma_N$ corresponding to masses $m_1, m_i $ and $m_N$ are shown in dashed line segments for emphasis. }\label{Fig2}
\end{figure}

\par Moreover, sticky particle trajectories can be used to generate a solution pair $\rho$ and $v$ of the SPS. Indeed, we may define the function $\rho=\rho_t$ which takes values in the space of Borel measures on $\R$ via
\be\label{discreteRho}
\rho_t=\sum^N_{i=1}m_i\delta_{\gamma_i(t)},\quad t\ge 0.
\ee
Note that $\rho$ is the mass distribution of the particles as $\rho_t(A)$ is the amount of mass within the set $A\subset \R$ at time $t\ge 0$. We can also set 
\be\label{discreteVee}
v(x,t)=
\begin{cases}
\dot \gamma_i(t+), \quad & x=\gamma_i(t)\\
0, \quad & \text{otherwise}.
\end{cases}
\ee
We note that $v:\R\times[0,\infty)\rightarrow \R$ is Borel measurable and $v(\gamma_i(t),t)$ is the right hand slope of the particles located at position $\gamma_i(t)$ at time $t$. 

% Weak solution
\par While $\rho$ and $v$ are not smooth functions, they turn out to satisfy the SPS in a certain sense that we will specify below.   As we expect the total mass to be conserved for all times, we will assume that it is always equal to 1 for convenience.  Consequently, it will be natural for us to work with the space ${\cal P}(\R)$ of Borel probability measures on $\R$.   We recall this space has a natural topology: $(\mu^k)_{k\in \N}\subset {\cal P}(\R)$ converges to $\mu$ {\it narrowly} provided 
\be
\lim_{k\rightarrow\infty}\int_{\R}gd\mu^k=\int_{\R}gd\mu
\ee
for each bounded, continuous $g:\R\rightarrow \R$.

% Defn
\begin{defn}\label{WeakSolnDefn}
Suppose $\rho_0\in {\cal P}(\R)$ and $v_0:\R\rightarrow \R$ is continuous with 
$$
\int_{\R}|v_0|d\rho_0<\infty.
$$
A narrowly continuous 
$\rho: [0,\infty)\rightarrow {\cal P}(\R); t\mapsto \rho_t$ and a Borel measurable
$v:\R\times[0,\infty)\rightarrow\R$ is a {\it weak solution pair of the sticky particle system}
with the initial conditions \eqref{Init} if the following conditions hold. \\
$(i)$ For each $T>0$,
$$
\int^T_0\int_{\R}v^2d\rho_tdt<\infty.
$$ 
$(ii)$ For each $\phi\in C^\infty_c(\R\times[0,\infty))$,
$$
\int^\infty_0\int_{\R}(\partial_t\phi+v\partial_x\phi)d\rho_tdt+\int_{\R}\phi(\cdot,0)d\rho_0=0.
$$
$(iii)$ For each $\phi\in C^\infty_c(\R\times[0,\infty))$,
$$
\int^\infty_0\int_{\R}(v\partial_t\phi +v^2\partial_x\phi)d\rho_tdt+\int_{\R}\phi(\cdot,0)v_0d\rho_0=0.
$$
\end{defn}

% Brief Literature Review 
\par It can be shown that the pair $\rho$ and $v$ specified in \eqref{discreteRho} and \eqref{discreteVee} is indeed a weak solution pair with initial mass 
$$
\rho_0=\sum^N_{i=1}m_i\delta_{\gamma_i(0)}
$$
and initial velocity $v_0: \R\rightarrow \R$ chosen to satisfy 
$$
v_0(\gamma_i(0))=\dot \gamma_i(0+)
$$
for $i=1,\dots, N$.  A challenging problem is to show that there is a solution for a general set of initial conditions. This was first accomplished by E, Rykov and Sinai \cite{ERykovSinai} who identified a variational principle for the SPS. Around the same the time, Brenier and Grenier established a general existence theory by reinterpreting the SPS as a single scalar conservation law \cite{BreGre}.   These two approaches appeared to be distinct until they were merged and extended upon by Natile and Savar\'e \cite{NatSav}; see also Cavalletti, Sedjro and Westdickenberg's paper \cite{MR3296602} for a refinement of \cite{NatSav}.  In addition, we mention that these approaches are relevant to the dynamics of collections of sticky particles with more general pairwise interactions as discussed in \cite{BreGan, GNT, MR2438785,MR3359159}.

% Our approach 
\par In this work, we will consider Lagrangian coordinates for the sticky particle system as motivated by a probabilistic approach introduced by Dermoune \cite{Dermoune}. This involves finding an absolutely continuous mapping $X: [0,\infty)\rightarrow L^2(\rho_0)$ which satisfies {\it the sticky particle flow equation}
\be\label{FlowEqn}
\dot X(t)=\E_{\rho_0}[v_0| X(t)]\;\; \text{a.e.}\;\;t \ge 0
\ee
and initial condition
\be\label{FlowInit}
X(0)=\id
\ee
$\rho_0$ almost everywhere. Here $\E_{\rho_0}[v_0| X(t)]$ is the conditional expectation of $v_0: \R\rightarrow \R$ with respect to $\rho_0$ given $X(t)$.  In particular, we are asserting that \eqref{FlowEqn} is the natural condition for collections of particles that move freely on the real line and undergo perfectly inelastic collisions when they meet. We note that Dermoune considered a more general setup involving an abstract probability space and showed the existence of a solution for a given initial condition. With regard to his formulation, we content ourselves with the specific probability space $(\R, {\cal B}(\R), \rho_0)$, where ${\cal B}(\R)$ is the Borel sigma algebra on $\R$. 

\par  We will also use the notation 
$$
X(t): \R\rightarrow \R; y\mapsto X(y,t)
$$
when we wish to emphasize spatial dependence. Here $X(y,t)$ denotes the position of the particle at time $t$ which started at position $y$. In particular, we will show that we can design a weak solution pair $\rho$ and $v$ of the SPS with
$$
\dot X(t)=v(X(t),t)\;\; \text{a.e.}\;\;t \ge 0.
$$
In this sense, $X$ is a Lagrangian coordinate.  Our main theorem is as follows. 

% Main Theorem
\begin{thm}\label{mainThm}
Suppose $\rho_0\in{\cal P}(\R)$ with $$\int_{\R} x^2d\rho_0(x)<\infty$$ and $v_0:\R\rightarrow \R$ absolutely continuous. There is a solution $X$ of the sticky particle flow equation \eqref{FlowEqn} which satisfies the initial condition \eqref{FlowInit} and has the following properties. 
\begin{enumerate}[(i)]

\item For Lebesgue almost every $t,s\ge0$ with $s\le t$,
$$
\int_{\R}\frac{1}{2}\dot X(t)^2d\rho_0\le \int_{\R}\frac{1}{2}\dot X(s)^2d\rho_0\le\int_{\R}\frac{1}{2}v_0^2d\rho_0.
$$

\item For $t\ge 0$ and $y,z\in \textup{supp}(\rho_0)$ with $y\le z$, 
$$
0\le X(z,t)-X(y,t)\le z-y+t\int^z_y|v_0'(x)|dx.
$$

\item For each $0<s\le t$ and $y,z\in \textup{supp}(\rho_0)$,
$$
\frac{1}{t}|X(y,t)-X(z,t)|\le \frac{1}{s}|X(y,s)-X(z,s)|.
$$

\end{enumerate}
\end{thm}
\begin{rem}
Since $v_0:\R\rightarrow \R$ is absolutely continuous, it grows at most linearly on $\R$. As a result, $\displaystyle\int_{\R}v_0^2d\rho_0<\infty$.  We also remind the reader that the support of $\rho_0$ is defined
$$
\textup{supp}(\rho_0):=\{y\in \R: \rho_0((y-\delta,y+\delta))>0\; \text{for all}\;\delta>0\}.
$$
\end{rem}

% Old Theorem
\par A corollary of the above theorem is that there exists a weak solution of the SPS for given initial conditions.   We emphasize that the following result has already been proven or follows from previous efforts such as \cite{BreGre, ERykovSinai, NatSav}. Our goal is to verify this claim through proving Theorem \ref{mainThm} and in particular to give a more thorough analysis of \eqref{FlowEqn} than was done in \cite{Dermoune}.
\begin{cor}\label{oldThm}
Suppose $\rho_0\in{\cal P}(\R)$ with $$\int_{\R}x^2d\rho_0(x)<\infty$$
and $v_0:\R\rightarrow \R$ absolutely continuous.  There is a weak solution pair $\rho$ and $v$ of the SPS with initial conditions \eqref{Init}.   
\begin{enumerate}[(i)]

\item For Lebesgue almost every $t,s\ge 0$ with $s\le t$,  
$$
\int_{\R}\frac{1}{2}v(x,t)^2d\rho_t(x)\le \int_{\R}\frac{1}{2}v(x,s)^2d\rho_s(x)\le \int_{\R}\frac{1}{2}v_0(x)^2d\rho_0(x).
$$

\item For Lebesgue almost every $t\in (0,\infty)$,
\be\label{Entropy}
(v(x,t)-v(y,t))(x-y)\le \frac{1}{t}(x-y)^2
\ee
for $\rho_t$ almost every $x,y\in \R$. 
\end{enumerate}
\end{cor} 
We will prove this corollary at the end of this paper, right after verifying Theorem \ref{mainThm}. This paper is organized as follows. First, we will briefly discuss the preliminary material needed in our study and make some observations on sticky particle trajectories. Then we will verify that solutions of the sticky particle flow equation \eqref{FlowEqn} which are associated with sticky particle trajectories are compact in a certain sense. Finally, we will show that we can always find a subsequence of these particular types of solutions that converges to a general solution.

%%%%%%%%%%%%%%%%%%%%%%%%%%%%%%%%%%%%%%%%%%%%%%%%
\section{Preliminaries}\label{PreLimSect}
In this section, we will briefly outline some of the notation and review the few technical preliminaries needed for our study.

%--------------------------------------------------------------------------
\subsection{Convergence of probability measures}
We will denote ${\cal P}(\R^d)$ as the space of Borel probability measures on $\R^d$ and write $C_b(\R^d)$ for the space of bounded continuous functions on $\R^d$. As noted in the introduction, ${\cal P}(\R^d)$ is endowed with a natural topology defined as follows. A sequence $(\mu^k)_{k\in\N}\subset {\cal P}(\R^d)$ converges to $\mu$ in ${\cal P}(\R^d)$ {\it narrowly} provided 
\be\label{gConvUndermukay}
\lim_{k\rightarrow\infty}\int_{\R^d}gd\mu^k=\int_{\R^d}gd\mu
\ee
for each $g\in C_b(\R^d)$.  It turns out that this topology can be metrized by a metric of the form
\be\label{NarrowMetric}
\mathcal{d}(\mu,\nu):=\sum^\infty_{j=1}\frac{1}{2^j}\left|\int_{\R^d}h_jd\mu-\int_{\R^d}h_jd\nu\right|,\quad \mu,\nu\in {\cal P}(\R^d).
\ee
Here each $h_j:\R^d\rightarrow \R$ satisfies $\|h_j\|_\infty\le 1$ and Lip$(h_j)\le 1$ (Remark 5.1.1 of \cite{AGS}). Moreover, $({\cal P}(\R^d), \mathcal{d})$ is a complete metric space. 

\par It will be useful for us to know when a sequence of measures in ${\cal P}(\R^d)$ has a narrowly convergent subsequence.  Prokhorov's theorem asserts that  $(\mu^k)_{k\in\N}\subset {\cal P}(\R^d)$ has a narrowly convergent subsequence if and only if there is $\varphi: \R^d\rightarrow [0,\infty]$ with compact sublevel sets for which
\be\label{prohorovCond}
\sup_{k\in \N}\int_{\R^d}\varphi d\mu^k<\infty
\ee
(Theorem 5.1.3 of \cite{AGS}). It will also be convenient to know when \eqref{gConvUndermukay} holds for unbounded $g$. It turns out that if $g: \R^d\rightarrow \R$ is continuous and  $|g|$ is {\it uniformly integrable} with respect to $(\mu^k)_{k\in\N}$ then \eqref{gConvUndermukay} holds. That is, provided
$$
\lim_{R\rightarrow\infty}\int_{|g|\ge R}|g|d\mu^k=0
$$
uniformly in $k\in \N$ (Lemma 5.1.7 of \cite{AGS}).    

\par We will also need the following lemma. 

\begin{lem}\label{LemmaVariantNarrowCon}
Suppose $(g^k)_{k\in\N}$ is a sequence of continuous functions on $\R^d$ which converges locally uniformly to $g$ and $(\mu^k)_{k\in\N}\subset {\cal P}(\R^d)$ converges narrowly to $\mu$. Further assume there is $h:\R^d\rightarrow [0,\infty)$ with compact sublevel sets, which is uniformly integrable with respect to  $(\mu^k)_{k\in\N}$ and  satisfies
\be\label{gkaydominatedbyH}
|g^k|\le h
\ee
for each $k\in \N$. Then
\be\label{UniformonvUndermukay}
\lim_{k\rightarrow\infty}\int_{\R^d}g^kd\mu^k=\int_{\R^d}gd\mu.
\ee
\end{lem}
\begin{proof} 
Fix $\epsilon>0$ and choose $R>0$ so large that 
$$
\int_{h\ge R}hd\mu^k\le \frac{\epsilon}{4}
$$   
for all $k\in \N$. In view of \eqref{gkaydominatedbyH}, $|g|\le h$. Thus, $|g|$ is uniformly integrable with respect to $(\mu^k)_{k\in\N}$ and so there is $N\in \N$ such that
$$
\left|\int_{\R^d}gd\mu^k-\int_{\R^d}gd\mu\right|<\frac{\epsilon}{2}
$$
for all $k\ge N$. 

\par It follows that
\begin{align*}
\left|\int_{\R^d}g^kd\mu^k-\int_{\R^d}gd\mu\right|&=\left|\int_{\R^d}(g^k-g)d\mu^k+\int_{\R^d}gd\mu^k-\int_{\R^d}gd\mu\right|\\
&\le \int_{\R^d}|g^k-g|d\mu^k+\frac{\epsilon}{2}\\
&= \int_{h\le R}|g^k-g|d\mu^k+\int_{h\ge R}|g^k-g|d\mu^k+\frac{\epsilon}{2}\\
&\le \int_{h\le R}|g^k-g|d\mu^k+2\int_{h\ge R}hd\mu^k+\frac{\epsilon}{2}\\
&\le \int_{h\le R}|g^k-g|d\mu^k+\epsilon
\end{align*} 
for $k\ge N$.  As $\{h\le R\}$ is compact and $g^k\rightarrow g$ uniformly on $\{h\le R\}$,
$$
\limsup_{k\rightarrow\infty}\left|\int_{\R^d}g^kd\mu^k-\int_{\R^d}gd\mu\right|\le\epsilon. 
$$
We conclude \eqref{UniformonvUndermukay} as $\epsilon>0$ is arbitrary. 
\end{proof}

%-----------------------------------------------------------------
\subsection{The push-forward}
For a Borel map $f:\R^d\rightarrow \R^n$ and $\mu\in {\cal P}(\R^d)$, we define the {\it push-forward} of $\mu$ through $f$ as the probability measure $f_{\#}\mu\in {\cal P}(\R^n)$ which satisfies 
$$
\int_{\R^n}g(y)d(f_{\#}\mu)(y)=\int_{\R^d}g(f(x))d\mu(x)
$$
for each $g\in C_b(\R^n)$.  We also note 
$$
f_{\#}\mu(A)=\mu(f^{-1}(A))
$$
for Borel $A\subset \R^n$.  
\begin{rem}
$(i)$ We will be primarily interested in the dimensions $d,n\in \{1,2\}$. $(ii)$ We could have easily have presented our remarks involving the convergence of probability measures and the push-forward in terms of complete, separable metric spaces instead of focusing on Euclidean spaces.
\end{rem}

%--------------------------------------------------------------------------
\subsection{Conditional expectation}
Suppose $\mu\in {\cal P}(\R)$, $g\in L^2(\mu)$ and $Y: \R\rightarrow \R$ is Borel measurable.
A {\it conditional expectation} of $g$ with respect to $\mu$ given $Y$ is an $L^2(\mu)$ function $\E_{\mu}[g|Y]$ which satisfies
\be\label{IntegralCondExpCond}
\displaystyle\int_{\R}\E_\mu[g|Y]\; h(Y)d\mu=\int_{\R}g\; h(Y)d\mu
\ee
for all Borel $h:\R\rightarrow \R$ with 
\be\label{hSqIntegrable}
\int_{\R}h(Y)^2d\mu<\infty
\ee
and 
$$
\displaystyle\E_{\mu}[g|Y]=f(Y)\quad \mu\;\text{a.e.}
$$
for some Borel $f:\R\rightarrow \R$ which satisfies \eqref{hSqIntegrable} (with $f$ replacing $h$).  

\par The existence of a conditional expectation follows from a simple application of the Radon-Nikodym theorem, and it is also not hard to show that conditional expectations are uniquely determined up to a null set for $\mu$. Moreover, choosing $h(Y)=\E_\mu[g|Y]$ in \eqref{IntegralCondExpCond} and using the Cauchy-Schwarz inequality gives 
\be\label{JensenINeq}
\int_{\R}\E_{\mu}[g|Y]^2d\mu\le \int_{\R}g^2d\mu.
\ee
Finally, we recall that conditional expectation has the ``tower property," which asserts  
\be\label{TowerProp}
\E_\mu[\E_{\mu}[g|Y] |e(Y)]=\E_{\mu}[g|e(Y)]
\ee
for any Borel $e:\R\rightarrow \R$.

% Sticky particle trajectories 
%%%%%%%%%%%%%%%%%%%%%%%%%%%%%%%%%%%%%%%%%%%%%%%%
\section{Sticky particle trajectories}\label{SPSsection}
We will now study the sticky particle trajectories mentioned in the introduction. To this end, we will fix $m_1,\dots, m_N>0$ with 
$$
\sum^N_{i=1}m_i=1,
$$
distinct $x_1,\dots, x_N\in \R$, and $v_1,\dots, v_N\in \R$ throughout this section.   These quantities represent the respective masses, initial positions and initial velocities of a collection of particles that will move freely and undergo perfectly inelastic collisions when they collide.   We will ultimately argue that we can always associate a collection of sticky particle trajectories $\gamma_1,\dots,\gamma_N$ to this initial data that has the necessary features in order to build a weak solution pair of the SPS out of them. 

%---------------------------------------------
\subsection{Basic properties}
We will first note that sticky particle trajectories exist. In the following proposition, we will use the notation 
$$
f(t\pm)=\lim_{h\rightarrow 0^\pm}f(t+h)
$$
for the right $f(t+)$ and left $f(t-)$ limits of $f$ at $t$, respectively. However, we will omit a proof of the following proposition as we have already justified this claim in a related work (Proposition 2.1 in \cite{Hynd}).  

% Existence of paths 
\begin{prop}\label{StickyParticlesExist}
There are continuous, piecewise linear paths 
$$
\gamma_1,\dots,\gamma_N : [0,\infty)\rightarrow \R
$$
with the following properties. \\
(i) For $i=1,\dots, N$,
$$
\gamma_i(0)=x_i\quad \text{and}\quad \dot\gamma_i(0+)=v_i.
$$
(ii) For $i,j=1,\dots, N$, $0\le s\le t$ and $\gamma_i(s)=\gamma_j(s)$ imply 
$$
\gamma_i(t)=\gamma_j(t).
$$
(iii) If $t>0$, $\{i_1,\dots, i_k\}\subset\{1,\dots, N\}$, and
$$
\gamma_{i_1}(t)=\dots=\gamma_{i_k}(t)\neq \gamma_i(t)
$$
for $i\not\in\{i_1,\dots, i_k\}$, then
$$
\dot\gamma_{i_j}(t+)=\frac{m_{i_1}\dot\gamma_{i_1}(t-)+\dots+m_{i_k}\dot\gamma_{i_k}(t-)}{m_{i_1}+\dots+m_{i_k}}
$$
for $j=1,\dots, k$.
\end{prop}
\begin{rem}\label{interchangeSlopeLimit}
Since $\gamma_i$ is piecewise linear, the limits $\dot\gamma_i(t+)$ and $\dot\gamma_i(t-)$ exist. Moreover, they can be computed as follows
$$
\dot\gamma_i(t\pm)=\lim_{h\rightarrow 0^\pm}\frac{\gamma_i(t+h)-\gamma_i(t)}{h}.
$$
\end{rem}
We also note that property $(iii)$ implies a more general {\it averaging property}, which is stated below. This is the main tool that can be used to show that $\rho$ and $v$ defined in \eqref{discreteRho} and \eqref{discreteVee} constitute a weak solution pair of the SPS. We will omit the proof of this fact as we have verified it in earlier work (Proposition 2.5 in \cite{Hynd}). 
% > Averaging 
\begin{cor}\label{averagingCorollary}
For $g:\R\rightarrow \R$ and $0\le s\le t$,
\be\label{AveragingProp}
\sum^N_{i=1}m_ig(\gamma_i(t))\dot\gamma_i(t+)= \sum^N_{i=1}m_ig(\gamma_i(t))\dot\gamma_i(s+).
\ee
\end{cor}

%---------------------------------------------
\subsection{Two estimates}
We will now derive some estimates on $\gamma_i(t)-\gamma_j(t)$ in terms of the given initial data.  We will start with an elementary lemma. 

% > Modulus of continuity estimate
\begin{lem}\label{YConcaveLemma}
Suppose $T>0$ and $y:[0,T)\rightarrow \R$ is continuous and piecewise linear. Further assume 
\be\label{slopeDecreaseY}
\dot y(t+)\le \dot y(t-)
\ee
for each $t\in (0,T)$.  Then 
\be\label{yConcave0Tee}
y(t)\le y(0)+t\dot y(0+).
\ee
for $t\in [0,T)$.
\end{lem}
\begin{proof}
Choose times $0=t_0<\dots <t_n=T$ such that $y$ is linear on each of the intervals $(0,t_1), \dots, (t_{n-1},T)$.
For $\phi \in C^\infty_c(0,T)$, we integrate by parts and compute 
\begin{align*}
\int^T_0\ddot \phi(t) y(t)dt&=\sum^{n-1}_{i=1}\phi(t_i)\left(\dot y(t_i+)-\dot y(t_i-)\right)+\sum^{n-1}_{i=1}\int^{t_{i+1}}_{t_i} \phi(t) \ddot y(t)dt\\
&=\sum^{n-1}_{i=1}\phi(t_i)\left(\dot y(t_i+)-\dot y(t_i-)\right).
\end{align*}
Thus, 
\be\label{ConcaveinDist}
\int^T_0\ddot \phi(t) y(t)dt\le 0
\ee
for $\phi\ge 0$. 

\par Now let $\eta\in C^\infty_c(\R)$ be a standard mollifier. That is, 
$$
\begin{cases}
\eta(z)=\eta(-z)\ge 0\\
\int_{\R}\eta dz=1\\
\text{supp}(\eta)\subset [-1,1].
\end{cases}
$$ 
Set 
$$
\eta^\epsilon:=\frac{1}{\epsilon}\eta\left(\frac{\cdot}{\epsilon}\right)
$$
and define
$$
y^\epsilon(s)=\int^T_{0}\eta^\epsilon(s-t)y(t)dt
$$
for $s\in (\epsilon, T-\epsilon)$ and $0<\epsilon<T$.  Observe that $y^\epsilon$ is smooth and $\eta^\epsilon(s-\cdot)\in C^\infty_c(0,T)$ for $s\in (\epsilon, T-\epsilon)$. By \eqref{ConcaveinDist},
$$
\ddot y^\epsilon(s)=\int^T_{0}\ddot \eta^\epsilon(s-t)y(t)dt\le 0.
$$
Therefore, $y^\epsilon$ is concave on $(\epsilon, T-\epsilon)$ for any $0<\epsilon<T$. It is routine to check that $y^\epsilon(t)\rightarrow y(t)$ for each $t\in (0,T)$. As a result, $y$ is concave on $[0,T)$ and we conclude \eqref{yConcave0Tee}.  
\end{proof}
The main application of Lemma \ref{YConcaveLemma} is the following proposition. It will later provide us with a modulus of continuity estimate for solutions of \eqref{FlowEqn}. 

% continuity estimate 
\begin{prop}\label{ContPropgamma}
Suppose $i,j\in\{1,\dots, N\}$, $x_i\ge x_j$ and $t\ge 0$. Then 
\be\label{timeZeroEst}
0\le \gamma_i(t)-\gamma_j(t)\le x_i-x_j+t\sum^{n-1}_{\ell=1}|v_{k_{\ell+1}}-v_{k_\ell}|
\ee
where $k_1,\dots,k_n\in \{1,\dots, N\}$ are chosen so that 
$$
x_j=x_{k_1}<\dots<x_{k_n}=x_i.
$$
\end{prop}
\begin{proof}
1. We suppose $x_1\le\dots\le x_N$ so that $\gamma_1\le \dots\le\gamma_N$. With this assumption, it suffices to show 
\be\label{timeZeroEstiplusone}
\gamma_{i+1}(t)-\gamma_i(t)\le x_{i+1}-x_i+t|v_{i+1}-v_i|
\ee
for $t\ge 0$.  Because if $j,k\in \{1,\dots, N\}$ with $k>j$, then
\begin{align}
\gamma_k(t)-\gamma_j(t)&= \sum^{k-1}_{i=j}(\gamma_{i+1}(t)-\gamma_i(t))\\
&\le \sum^{k-1}_{i=j}\left(x_{i+1}-x_i+t|v_{i+1}-v_i|\right)
\\
&= x_k-x_j+t\sum^{k-1}_{i=j}|v_{i+1}-v_i|.
\end{align}

\par With the goal of verifying \eqref{timeZeroEstiplusone} in mind, we fix $i\in\{1,\dots, N\}$ and define
$$
T:=\inf\{t\ge 0: \gamma_{i+1}(t)-\gamma_i(t)=0\}.
$$
In order to prove \eqref{timeZeroEstiplusone}, it then suffices to show
\be\label{LasttimeZeroEst}
\gamma_{i+1}(t)-\gamma_i(t)\le x_{i+1}-x_i+t(v_{i+1}-v_i),\quad t\in [0,T].
\ee
We will do this by applying to the previous lemma to 
$$
y(t):=\gamma_{i+1}(t)-\gamma_i(t), \quad t\in [0,T).
$$
We already know that $y$ is continuous and piecewise linear. Let us now focus on showing
\be\label{DotGmamaiplusone}
\dot\gamma_{i+1}(s+)\le \dot\gamma_{i+1}(s-).
\ee

\par  2. Observe that if $\gamma_{i+1}$ does not have a first intersection time at $s\in (0,T)$, then $\gamma_{i+1}$ is linear near $s$ and so 
$$
\dot\gamma_{i+1}(s)=\dot\gamma_{i+1}(s+)=\dot\gamma_{i+1}(s-).
$$
Alternatively, if $\gamma_{i+1}$ has a first intersection time at $s$ there are trajectories $\gamma_{i+2}, \dots, \gamma_{i+r}$  (some $r\ge 2$) such that 
$$
\gamma_{i+1}(s)=\gamma_{i+2}(s)=\dots=\gamma_{i+r}(s)
$$
and 
\be\label{AveragingForQSPS}
\dot\gamma_{i+j}(s+)=\frac{m_{i+1}\dot\gamma_{i+1}(s-)+\dots + m_{i+r}\dot\gamma_{i+r}(s-)}{m_{i+1}+\dots+m_{i+r}}
\ee
$j=1,\dots, r$. Recall part $(iii)$ of Proposition \ref{StickyParticlesExist}.

\par Also observe that since $\gamma_{i+1}\le \gamma_{i+j}$ for $j=2,\dots, r$,
$$
\frac{\gamma_{i+1}(s+h)-\gamma_{i+1}(s)}{h}\ge \frac{\gamma_{i+j}(s+h)-\gamma_{i+j}(s)}{h}
$$
for all $h<0$ and close enough to 0. It follows from Remark \ref{interchangeSlopeLimit} that
$$
\dot\gamma_{i+1}(s-)\ge \dot\gamma_{i+j}(s-).
$$
It view of \eqref{AveragingForQSPS} 
 $$
\dot \gamma_{i+1}(s+)\le\frac{m_{i+1}\dot\gamma_{i+1}(s-)+\dots + m_{i+r}\dot\gamma_{i+1}(s-)}{m_{i+1}+\dots+m_{i+r}}=\dot\gamma_{i+1}(s-),
 $$
  which is \eqref{DotGmamaiplusone}. A similar argument gives 
\be\label{DotGmamajusti}
\dot\gamma_{i}(s+)\ge \dot\gamma_{i}(s-)
\ee
for each $s\in (0,T)$. Combining \eqref{DotGmamaiplusone} and \eqref{DotGmamajusti} 
 \be\label{dotYnotgoingup}
 \dot y(s+)=\dot\gamma_{i+1}(s+)-\dot\gamma_{i}(s+)\le \dot\gamma_{i+1}(s-)-\dot\gamma_{i}(s-)= \dot y(s-)
 \ee
for all $s\in (0,T)$.  We then conclude \eqref{LasttimeZeroEst} by appealing to Lemma \ref{YConcaveLemma}. 
\end{proof}
\begin{rem}\label{EstimateFailsForNonMonv}
We can infer from proof of Proposition \ref{ContPropgamma} that if $x_1\le \dots \le x_N$ and $v_1\le \dots \le v_N$, then \eqref{timeZeroEst} can be improved to 
\be\label{timeZeroEst2}
0\le \gamma_i(t)-\gamma_j(t)\le x_i-x_j+t|v_{i}-v_{j}|
\ee
for $i\ge j$ and $t\ge 0$. However, if $v_1,\dots, v_N$ are not nondecreasing then this estimate fails to be true.   To see this, let us consider the example of three particles each with mass equal to 1/3, and with respective initial positions
$$
x_1=0, x_2=1,x_3=2,
$$
and the initial velocities
$$
v_1=1, v_2=0,v_3=1.
$$
\begin{figure}[h]
\centering
 \includegraphics[width=.65\textwidth]{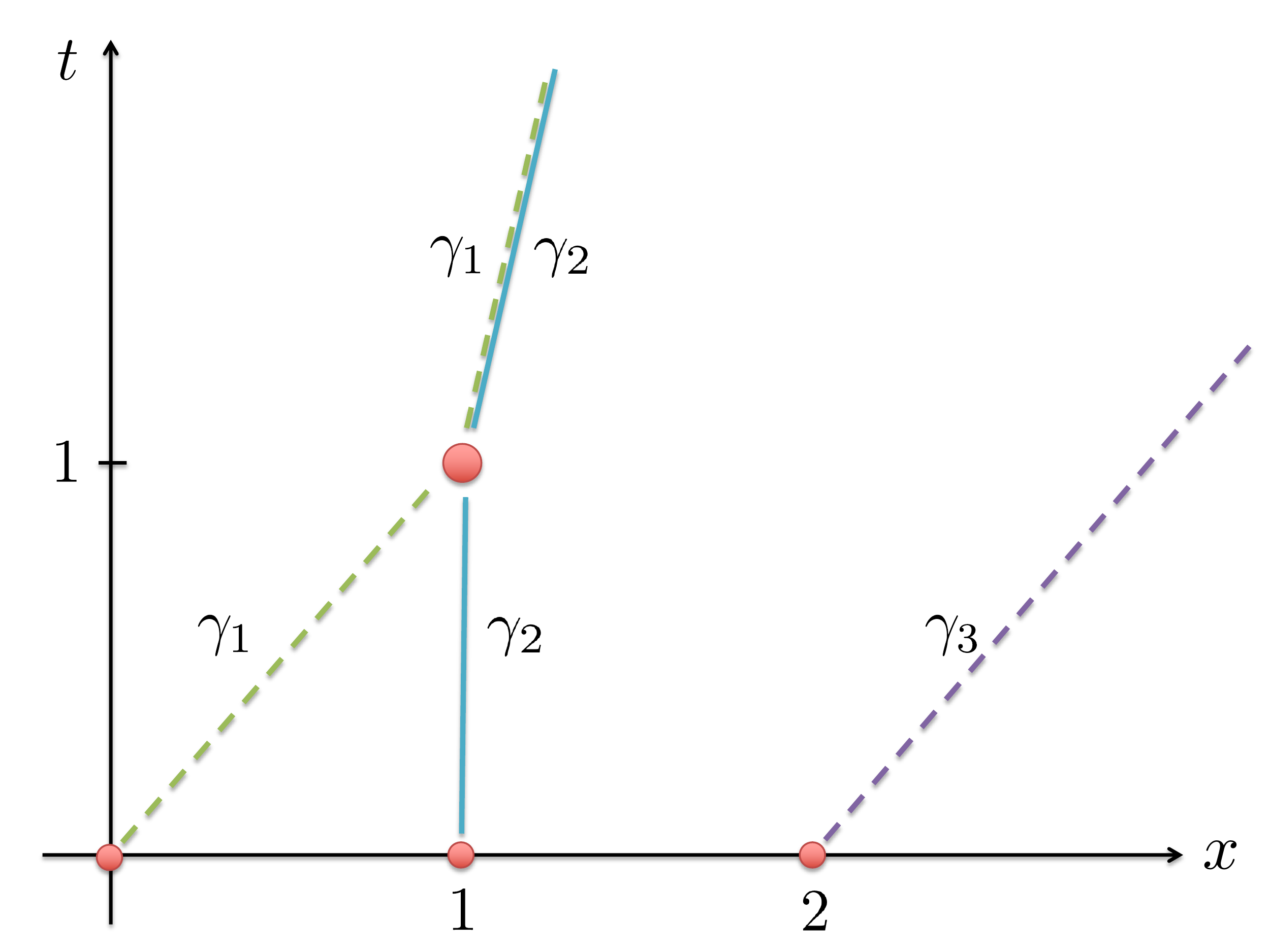}
 \caption{Three sticky particle trajectories $\gamma_1, \gamma_2, \gamma_3: [0,\infty)\rightarrow \R$ in which $|\gamma_1(t)-\gamma_3(t)|>|x_3-x_1|+t|v_3-v_1|$ for $t>1$. }\label{Fig3}
\end{figure}
\par 
The corresponding sticky particle trajectories for $\gamma_1$ are and $\gamma_3$ are
$$
\gamma_1(t)=
\begin{cases}
t,\quad &0\le t\le 1\\
1+\frac{1}{2}(t-1),\quad &t\ge 1
\end{cases}
$$
and $\gamma_3(t)=2+t$.  Observe that for $t> 1$ 
\begin{align*}
\gamma_3(t)-\gamma_1(t)&=2+t - \left(1+\frac{1}{2}(t-1)\right)\\
&=2+\frac{1}{2}(t-1)\\
&>2\\
&=x_3-x_1+t|v_3-v_1|.
\end{align*}
See Figure \ref{Fig3}. 
\end{rem}
% > Quantitative stickiness  
We call the following assertion the {\it quantitative sticky particle property} as it quantifies part $(ii)$ of Proposition \ref{StickyParticlesExist}. 

\begin{prop}\label{PropQSPP}
For each $i,j=1,\dots,N$ and $0<s\le t$,
\be
\frac{1}{t}|\gamma_i(t)-\gamma_j(t)|
\le \frac{1}{s}|\gamma_i(s)-\gamma_j(s)|.
\ee
\end{prop}
We will see that this proposition is a simple consequence of the following lemma. 
\begin{lem}\label{slopeDecreaseYLemma}
Suppose $T>0$ and $y:[0,T)\rightarrow [0,\infty)$ is continuous and piecewise linear. Further assume 
\be\label{slopeDecreaseY}
\dot y(t+)\le \dot y(t-)
\ee
for each $t\in (0,T)$.  Then 
\be\label{yoverteeNonincrease}
\frac{1}{t}y(t)\le\frac{1}{s}y(s) 
\ee
for $0<s\le t<T$.
\end{lem}

\begin{proof} Let $0<t_0<\dots < t_n<T$ be such that $y$ is linear on each of the intervals 
$(0,t_1),\dots (t_n,T)$.  It then suffices to show  
$$
u(t):=\frac{y(t)}{t}, \quad t\in (0,T)
$$
is nonincreasing on each of these intervals. First observe
\begin{align*}
\dot u(t+)&=\frac{\dot y(t+)}{t}-\frac{y(t)}{t^2}\\
&\le \frac{\dot y(t-)}{t}-\frac{y(t)}{t^2}\\
&=\dot u(t-)
\end{align*}
for each $t>0$ by \eqref{slopeDecreaseY}. Also note 
\be
\ddot y(t)  = t\ddot u(t)+2\dot u(t)= 0
\ee
for $t\in (0,T)\setminus\{t_1,\dots, t_n\}$. Consequently, 
\be\label{UdotLessThan}
\frac{d}{dt}\left(\dot u(t)t^2\right)=t\left(\ddot u(t)t+2\dot u(t)\right)= 0
\ee
for $t\in (0,T)\setminus\{t_1,\dots, t_n\}$. 

\par 
As $y$ is nonnegative, 
$$
\dot u(t)\le \frac{\dot y(t)}{t}
$$
for $t\in (0,T)\setminus\{t_1,\dots, t_n\}$.  As a result, 
$$
\limsup_{s\rightarrow 0^+}\left\{\dot u(s)s^2\right\}\le \limsup_{s\rightarrow 0^+}\left\{\dot y(s)s\right\}=\dot y(0+)0=0.
$$
In view of \eqref{UdotLessThan},
$$
\dot u(t)t^2= \limsup_{s\rightarrow 0^+}\left\{\dot u(s)s^2\right\}\le 0
$$
for $t\in (0,t_1)$.  Therefore, $\dot u(t)\le 0$ for $t\in (0,t_1)$. 

\par So far, we have shown that $u$ is nonincreasing on $t\in (0,t_1]$ which gives 
$$
\dot u(t_1-)\le 0.
$$
And by \eqref{UdotLessThan}, 
\begin{align*}
\dot u(t)t^2&\le \dot u(t_1+)t_1^2\\
&\le \dot u(t_1-)t_1^2\\
&\le 0
\end{align*}
for $t\in (t_1,t_2)$.  Thus, $u$ is nonincreasing on $[t_1,t_2]$ and 
$$
\dot u(t_2-)\le 0.
$$
Repeating this argument on $[t_2,t_3], [t_3,t_4],\dots, [t_n,T)$, we find $u$ is nonincreasing on $(0,T)$.
\end{proof}
\begin{proof}[Proof of Proposition \ref{PropQSPP}]
Without loss of generality, we may suppose $\gamma_1\le \dots\le \gamma_N$.  Then it suffices to show
\be\label{StepNeededQSPP}
\frac{1}{t}(\gamma_{i+1}(t)-\gamma_i(t))\le \frac{1}{s}(\gamma_{i+1}(s)-\gamma_i(s))
\ee
for each $i=1,\dots, N-1$ and $0<s\le t<\infty$.  In this case, we would have for $k>j$ 
\begin{align*}
\frac{1}{t}(\gamma_{k}(t)-\gamma_j(t))&=\sum^{k-1}_{i=j}\frac{1}{t}(\gamma_{i+1}(t)-\gamma_i(t))\\
&\le\sum^{k-1}_{i=j}\frac{1}{s}(\gamma_{i+1}(s)-\gamma_i(s))\\
&=\frac{1}{s}(\gamma_{k}(s)-\gamma_j(s)).
\end{align*}

\par As for \eqref{StepNeededQSPP}, we fix $i\in \{1,\dots, N-1\}$ and set 
$$
y(t):=\gamma_{i+1}(t)-\gamma_i(t), \quad t\in [0,T).
$$
Here 
$$
T:=\inf\{t\ge 0: y(t)=0\}.
$$
Observe that $y:[0,T)\rightarrow [0,\infty)$ is piecewise linear. Further, the proof of Proposition \ref{ContPropgamma}
gives 
$$
\dot y(t+)\le \dot y(t-).
$$ 
By Lemma \ref{slopeDecreaseYLemma}, 
$$
\frac{1}{t}(\gamma_{i+1}(t)-\gamma_i(t))=\frac{1}{t}y(t)\le\frac{1}{s}y(s)
=\frac{1}{s}(\gamma_{i+1}(s)-\gamma_i(s))
$$ 
for $0<s\le t<T$. Since $\gamma_{i+1}(\tau)-\gamma_i(\tau)=0$ for all $\tau\ge T$, we conclude \eqref{StepNeededQSPP} for all $0<s\le t<\infty$. 
\end{proof}
% Transition functions
\begin{cor}\label{TransitionFunction}
For each $0<s\le t$ there is $f_{t,s}:\R\rightarrow \R$ such that 
$$
\gamma_i(t)=f_{t,s}(\gamma_i(s))
$$
for $i=1,\dots, N$ and 
\be\label{ftsLip}
|f_{t,s}(x)-f_{t,s}(y)|\le \frac{t}{s}|x-y|
\ee
for $x,y\in \R$.
\end{cor}
\begin{proof}
Since the number of distinct elements of $\{\gamma_1(\tau),\dots,\gamma_N(\tau)\}$ is nonincreasing in $\tau\ge 0$, the function
$$
g_{t,s}: \{\gamma_1(s),\dots,\gamma_N(s)\}\rightarrow\{\gamma_1(t),\dots,\gamma_N(t)\}; \gamma_i(s)\mapsto \gamma_i(t)
$$
is well defined by part $(ii)$ of Proposition \ref{StickyParticlesExist}. Further, $\gamma_i(t)=g_{t,s}(\gamma_i(s))$ for $i=1,\dots, N$.
By Proposition \ref{PropQSPP}, $g_{t,s}$ satisfies \eqref{ftsLip} for $x,y\in\{\gamma_1(s),\dots,\gamma_N(s)\}$. We can then extend $g_{t,s}$ to all of $\R$ to obtain $f_{t,s}$. For example, we can take 
$$
f_{t,s}(x)=\inf\left\{g_{t,s}(\gamma_i(s))+\frac{t}{s}|x-\gamma_i(s)|: i=1,\dots, N\right\}.
$$
\end{proof}

% > Flow maps that are generated by sticky particle trajectories 
%---------------------------------------------
\subsection{A trajectory map}
Let us define 
$$
X: \{x_1,\dots, x_N\}\times[0,\infty)\rightarrow \R; (x_i,t)\mapsto \gamma_i(t).
$$
For each $t\ge 0$, we will also set 
$$
X(t): \{x_1,\dots, x_N\}\rightarrow \R; x_i\mapsto \gamma_i(t)
$$
so that $$X(t)(x_i)=X(x_i,t)=\gamma_i(t)$$ for $i=1,\dots, N$.  This is a trajectory map associated with the sticky particle trajectories $\gamma_1,\dots,\gamma_N$.   

\par We will translate the properties we derived above for sticky particle trajectories in terms of $X$ and argue that $X$ is a solution of the sticky particle flow equation \eqref{FlowEqn}.  To this end, we set 
\be\label{rhozeroConvComb}
\rho_0:=\sum^N_{i=1}m_i\delta_{x_i}
\ee
and choose $v_0:\R\rightarrow \R$ absolutely continuous with 
$$
v_0(x_i)=v_i
$$
for $i=1,\dots, N$.  
\begin{prop}\label{DiscreteSolnProp}
The function $X$ has the following properties. 

\begin{enumerate}[(i)]

\item $X(0)=\textup{id}_\R$ and 
\be
\dot X(t)=\E_{\rho_0}[v_0| X(t)]
\ee
for all but finitely many $t\ge 0$. Both equalities hold on the support of $\rho_0$. 

\item For every $t,s\ge0$ with $s\le t$,
$$
\int_{\R}\frac{1}{2}\dot X(t+)^2d\rho_0\le \int_{\R}\frac{1}{2}\dot X(s+)^2d\rho_0\le \int_{\R}\frac{1}{2}v_0^2d\rho_0.
$$

\item $X: [0,\infty)\rightarrow L^2(\rho_0); t\mapsto X(t)$ is Lipschitz continuous.

\item For $t\ge 0$ and $y,z\in\textup{supp}(\rho_0)$ with $y\le z$, 
$$
0\le X(z,t)-X(y,t)\le z-y+t\int^z_y|v_0'(x)|dx.
$$

\item For each $0<s\le t$ and $y,z\in\textup{supp}(\rho_0)$
$$
\frac{1}{t}|X(y,t)-X(z,t)|\le \frac{1}{s}|X(y,s)-X(z,s)|.
$$

\item For each $0<s\le t$, there is a function $f_{t,s}:\R\rightarrow \R$ which satisfies the Lipschitz condition \eqref{ftsLip} 
and
$$
X(y,t)=f_{t,s}(X(y,s))
$$
for $y\in\textup{supp}(\rho_0)$.
\end{enumerate}

\end{prop}
\begin{proof}
\underline{Part $(i)$}: As $X(x_i,0)=x_i$, it is clear that we have  
$$
X(0)=\id
$$
on the support of $\rho_0$. Furthermore, Corollary \ref{averagingCorollary} implies that if $g:\R\rightarrow \R$ and $s\le t$, then 
\begin{align*}
\int_{\R}g(X(t))\dot X(t+)d\rho_0&=\sum^N_{i=1}m_ig(\gamma_i(t))\dot\gamma_i(t+)\\
&=\sum^N_{i=1}m_ig(\gamma_i(t))\dot\gamma_i(s+)\\
&=\int_{\R}g(X(t))\dot X(s+)d\rho_0.
\end{align*}
In particular,  
$$
\int_{\R}g(X(t))\dot X(t)d\rho_0=\int_{\R}g(X(t)) v_0d\rho_0
$$
for all but finitely many $t\ge 0$. Also recall that 
$$
\dot X(t+)=v(X(t),t)
$$
on the support of $\rho_0$ for $t\ge 0$, where $v$ is defined in \eqref{discreteVee}. It follows that $\dot X(t)=\E_{\rho_0}[v_0| X(t)] $ for all but finitely many $t\ge 0$. 

\par \underline{Part $(ii)$ and $(iii)$}: Our proof of $(i)$ also shows that 
$$
\dot X(t+)=\E_{\rho_0}[\dot X(s+)| X(t)]
$$
and
$$
\dot X(s+)=\E_{\rho_0}[v_0| X(s)]
$$
for $0\le s\le t$. So part $(ii)$ follows from inequality \eqref{JensenINeq}. Moreover, for $s\le t$
\be\label{XLipschitzEst}
\int_{\R}(X(t)-X(s))^2d\rho_0\le (t-s)\int^t_s\int_{\R}\dot X(\tau)^2d\rho_0 d\tau\le (t-s)^2\int_{\R}v_0^2d\rho_0.
\ee
Therefore, $X:[0,\infty)\rightarrow L^2(\rho_0)$ is Lipschitz continuous. 

\par \underline{Part $(iv)$}: By part $(ii)$ of Proposition \ref{StickyParticlesExist}, $X(\cdot,t)$ is nondecreasing on the support of $\rho_0$.  In view of 
Proposition \ref{ContPropgamma}, we also have
\be
0\le X(x_i,t)-X(x_j,t)\le x_i-x_j+t\sum^{n-1}_{\ell=1}|v_0(x_{k_{\ell+1}})-v_0(x_{k_\ell})|
\ee
for $x_i\ge x_j$. Here $k_1,\dots,k_n\in \{1,\dots, N\}$ are chosen so that 
$$
x_j=x_{k_1}<\dots<x_{k_n}=x_i.
$$
Since $v_0$ is absolutely continuous, 
\be
\sum^{n-1}_{\ell=1}|v_0(x_{k_{\ell+1}})-v_0(x_{k_\ell})|\le 
\sum^{n-1}_{\ell=1}\int^{x_{k_{\ell+1}}}_{x_{k_\ell}}|v_0'(x)|dx =\int^{x_i}_{x_j}|v_0'(x)|dx.
\ee
We conclude part $(iv)$.
\par \underline{Part $(v)$ and $(vi)$}: Part $(v)$ follows from Proposition \ref{PropQSPP} and
part $(vi)$ is due to Corollary \ref{TransitionFunction}. 
\end{proof}
\begin{rem}\label{XdiscreteUnifContRem}
As $v_0:\R\rightarrow \R$ is absolutely continuous,  
$$
\omega(r):=\sup\left\{\int^b_a|v_0'(x)|dx \;:\; 0\le b-a\le r\right\}
$$
tends to $0$ as $r\rightarrow 0^+$.  It is also easy to check that $\omega$ is nondecreasing and sublinear, which implies that $\omega(r)$ grows at most linearly in $r$.  By part $(iv)$ of the above proposition, 
\be\label{XdiscreteUnifCont}
|X(y,t)-X(z,t)|\le |y-z|+t\omega(|y-z|)
\ee
for $y,z$ belonging to the support of $\rho_0$. Therefore, $X(t)$ is uniformly continuous on the support of $\rho_0$.  So we may extend $X(t)$ to obtain a uniformly continuous function on $\R$ which satisfies \eqref{XdiscreteUnifCont} and agrees with $X(t)$ on the support of $\rho_0$.  Consequently, we will identify $X(t)$ with this extension and consider $X(t)$ to be a uniformly continuous function on $\R$.
\end{rem}

\begin{rem}
The reader may wonder if the estimate 
$$
|X(z,t)-X(y,t)|\le |z-y|+t|v_0(y)-v_0(z)|
$$
holds for each $y,z$ belonging to the support of $\rho_0$. As we argued in Remark \ref{EstimateFailsForNonMonv}, such an estimate is only guaranteed to hold when $v_0$ is nonincreasing. 
\end{rem}

%%%%%%%%%%%%%%%%%%%%%%%%%%%%%%%%%%%%%%%%%%%%%%%%
\section{Existence theory}\label{CompactSect}
Our goal in this section is to prove Theorem \ref{mainThm}.  So we will assume throughout that $\rho_0\in {\cal P}(\R)$ with 
$$
\int_{\R}x^2d\rho_0(x)<\infty
$$
and $v_0:\R\rightarrow \R$ absolutely continuous. We will also select a sequence $(\rho^k_0)_{k\in \N}\subset {\cal P}(\R)$ in which each $\rho^k_0$ is of the form \eqref{rhozeroConvComb}, $\rho^k_0\rightarrow \rho_0$  narrowly and 
\be\label{rhokconveinpee2}
\lim_{k\rightarrow\infty}\int_{\R}x^2d\rho^k_0(x)= \int_{\R}x^2d\rho_0(x)
\ee
(see \cite{Bolley} for a short proof of how this can be done).  In view of Proposition \ref{DiscreteSolnProp}, there is a mapping 
$$
X^k:\R\times[0,\infty)\rightarrow \R
$$
which satisfies the sticky particle flow equation  \eqref{FlowEqn} and the initial condition \eqref{FlowInit} with $\rho^k_0$ replacing $\rho_0$. In this section, we will show $(X^k)_{k\in \N}$ has a subsequence that converges in various senses to a solution of the sticky particle flow equation \eqref{FlowEqn} which satisfies the initial condition \eqref{FlowInit} for the given $\rho_0$.  Then we will
finally show how to use this solution to design a solution of the SPS \eqref{SPS} that fulfills the initial conditions \eqref{Init}.

%---------------------------------------------------------------------
\subsection{Compactness}
Theorem \ref{mainThm} will follow from two compactness lemmas for the sequence $(X^{k})_{k\in \N}$. The first asserts that $X^k(t)$ has a subsequence which converges in a strong sense for each $t\ge 0$.
\begin{lem}\label{FirstCompactnessXkayJay}
There is a subsequence $(X^{k_j})_{j\in \N}$ and a Lipschitz continuous mapping $X:[0,\infty)\rightarrow L^2(\rho_0); t\mapsto X(t)$ such that 
\be\label{strongConvXkayJay}
\lim_{j\rightarrow\infty}\int_{\R}h( \textup{id}_\R,X^{k_j}(t))d\rho^{k_j}_0=\int_{\R}h( \textup{id}_\R,X(t))d\rho_0
\ee
for each $t\ge 0$ and continuous $h:\R^2\rightarrow \R$ with 
$$
\sup_{(x,y)\in \R^2}\frac{|h(x,y)|}{1+x^2+y^2}<\infty.
$$
Moreover, $X$ has the following properties.
\begin{enumerate}[(i)]  

\item For $y,z\in \textup{supp}(\rho_0)$ with $y\le z$ and $t\ge 0$,

\be\label{XYinequality}
0\le X(z,t)-X(y,t)\le z-y+t\int^z_y|v_0'(x)|dx.
\ee

\item For $y,z\in \textup{supp}(\rho_0)$ and $0<s\le t$,

\be\label{QSPPXinequality}
\frac{1}{t}|X(y,t)-X(z,t)|\le \frac{1}{s}|X(y,s)-X(z,s)|.
\ee

\item  For each $0<s\le t$, is there is a function $f_{t,s}:\R\rightarrow \R$ which satisfies \eqref{ftsLip} and 
\be\label{MeasureabilityforlimitX}
X(y,t)=f_{t,s}(X(y,s))
\ee
for $y\in  \textup{supp}(\rho_0)$. 
\end{enumerate} 

\end{lem}

\begin{proof} 
\underline{Step 1:  ``narrow" convergence}. Inequality \eqref{XLipschitzEst} implies
\begin{align*}
\left(\int_{\R}X^k(t)^2d\rho^k_0\right)^{1/2}
&\le\left(\int_{\R}(X^k(t)-X^k(0))^2d\rho^k_0\right)^{1/2}+\left(\int_{\R}X^k(0)^2d\rho^k_0\right)^{1/2}\\
&\le t \left(\int_{\R}v_0^2d\rho^k_0\right)^{1/2}+\left(\int_{\R}x^2d\rho^k_0\right)^{1/2}.
\end{align*}
As $v_0$ is uniformly continuous on $\R$, $v_0$ grows at most linearly. Combining with  \eqref{rhokconveinpee2}, we find 
\be\label{XkayTeeBoundedL2}
\left(\int_{\R}X^k(t)^2d\rho^k_0\right)^{1/2}\le A(1+t)
\ee
for some constant $A>0$ independent of $k\in \N$ and for each $t\ge 0$.  For $k\in \N$, we also define $\sigma^k:[0,\infty)\rightarrow {\cal P}(\R^2); t\mapsto \sigma^k_t$ via the formula
$$
\sigma^k_t:=(\id,X^k(t))_{\#}\rho_0^k.
$$
Note that \eqref{rhokconveinpee2} and \eqref{XkayTeeBoundedL2} give
\be\label{sigmakaytTight}
\sup_{k\in \N}\iint_{\R^2}(x^2+y^2)d\sigma^k_t(x,y)<\infty
\ee
for each $t\ge 0$.  By criterion \eqref{prohorovCond}, $(\sigma^k_t)_{k\in \N}$ is narrowly precompact for 
each $t\ge 0$.

\par Also observe that for $h\in C^1(\R^2)$ and $s\le t$
\begin{align*}
\iint_{\R^2}h(x,y)d\sigma^k_t(x,y)-\iint_{\R^2}h(x,y)d\sigma^k_s(x,y)&=\int_{\R}h(\id,X^k(t))d\rho^k_0-\int_{\R}h(\id,X^k(s))d\rho^k_0\\
&=\int_{\R}[h(\id,X^k(t))-h(\id,X^k(s))]d\rho^k_0\\
&=\int_{\R}\int^t_s\partial_yh(\id,X^k(\tau))\dot X^k(\tau)d\tau d\rho^k_0\\
&=\int^t_s\int_{\R}\partial_yh(\id,X^k(\tau))\dot X^k(\tau) d\rho^k_0 d\tau\\
%&=\int^t_s\int_{\R}\partial_yh(\id,X^k(\tau))v_0 d\rho^k_0 d\tau\\
&\le \text{Lip}(h)\int^t_s\left(\int_{\R}|\dot X^k(\tau)| d\rho^k_0 \right)d\tau\\
&\le \text{Lip}(h)(t-s)\int_{\R}|v_0| d\rho^k_0\\
&\le C\text{Lip}(h)(t-s)
\end{align*}
for some constant $C$ independent of $k\in \N$. By mollifying $h$, it is routine to show
$$
\left|\iint_{\R^2}h(x,y)d\sigma^k_t(x,y)-\iint_{\R^2}h(x,y)d\sigma^k_s(x,y)\right|\le C\text{Lip}(h)(t-s)
$$
for Lipschitz continuous $h: \R^2\rightarrow \R$.  

\par Using the metric defined in \eqref{NarrowMetric}, which metrizes the narrow topology on ${\cal P}(\R^2)$, we additionally have 
$$
\mathcal{d}(\sigma^k_t,\sigma^k_s)\le C|t-s|
$$
for $t,s\ge 0$ and $k\in \N$. In summary, $(\sigma^k)_{k\in \N}$ is a uniformly equicontinuous family of mappings from $[0,\infty)$ into $({\cal P}(\R^2),\mathcal{d})$ which is also pointwise precompact. By the Arzel\`a-Ascoli theorem, there is a subsequence $(\sigma^{k_j})_{j\in N}$ and a narrowly continuous mapping $\sigma:[0,\infty)\rightarrow {\cal P}(\R^2)$ such that 
\be\label{sigmakayjayConvPointwise}
\sigma^{k_j}_t\rightarrow \sigma_t
\ee
narrowly in ${\cal P}(\R^2)$ for each $t\ge 0$. 

\par \underline{Step 2:  ``weak" convergence}. A direct consequence of \eqref{sigmakayjayConvPointwise} is 
$$
\iint_{\R^2}\phi(x)d\sigma_t(x,y)=
\lim_{j\rightarrow\infty}\iint_{\R^2}\phi(x)d\sigma^{k_j}_t(x,y)=\lim_{j\rightarrow\infty}\int_{\R}\phi(x)d\rho^{k_j}_0(x)=\int_{\R}\phi(x) d\rho_0(x)
$$
for $\phi\in C_b(\R)$. By the disintegration theorem (Theorem 5.3.1 of \cite{AGS}), there is a family of probability measures $(\zeta^x_t)_{x\in \R}\subset {\cal P}(\R)$ such that 
$$
\iint_{\R^2}h(x,y)d\sigma_t(x,y)=\int_{\R}\left(\int_{\R}h(x,y)d\zeta^x_t(y)\right)d\rho_0(x)
$$
for $h\in C_b(\R^2)$.  We define 
$$
X(x,t):=\int_{\R} yd\zeta^x_t(y)
$$
for $x\in \R$ and $t\ge 0$. 
 
\par In view of \eqref{sigmakaytTight}, $(x,y)\mapsto |y|$ is uniformly integrable with respect to $(\sigma^{k_j}_t)_{j\in \N}$. Indeed,
$$
\iint_{|y|\ge R}|y|d\sigma^{k_j}_t(x,y)\le \frac{1}{R}\iint_{|y|\ge R}y^2d\sigma^{k_j}_t(x,y)\le 
\frac{1}{R}\iint_{\R^2}(x^2+y^2)d\sigma^{k_j}_t(x,y)
$$
so that 
$$
\lim_{R\rightarrow \infty}\iint_{|y|\ge R}|y|d\sigma^{k_j}_t(x,y)=0
$$
uniformly in $j\in \N$. It follows that
\begin{align}\label{weakLimXkayJay}
\lim_{j\rightarrow\infty}\int_{\R}X^{k_j}(t)\;\phi d\rho^{k_j}_0
&=\lim_{j\rightarrow\infty}\iint_{\R^2}y\;\phi(x)d\sigma^{k_j}_t(x,y)\\
&=\iint_{\R^2}y\;\phi(x)d\sigma_t(x,y)\\
&=\int_{\R}\left(\int_{\R}y\;\phi(x)d\zeta^x_t(y)\right)d\rho_0(x)\\
&=\int_{\R} X(t)\;\phi d\rho_0
\end{align}
for $\phi\in C_b(\R)$ and each $t\ge 0$.

\par \underline{Step 3: ``strong" convergence}.  Fix $t\ge 0$. By Remark \ref{XdiscreteUnifContRem}, 
$$
|X^k(y,t)-X^k(z,t)|\le |y-z|+t\omega(|y-z|)
$$
for $y,z\in \R$. Moreover, 
\begin{align*}
|X^k(y,t)|&\le |X^k(y,t)-X^k(z,t)|+|X^k(z,t)|\\
&\le |y-z|+t\omega(|y-z|)+|X^k(z,t)|.
\end{align*}
Integrating over $z\in \R$ gives
$$
|X^k(y,t)|\le \int_{\R}\left(|y-z|+t\omega(|y-z|)+|X^k(z,t)|\right)d\rho^k_0(z).
$$
In view of \eqref{rhokconveinpee2}, \eqref{XkayTeeBoundedL2},  and the fact that $\omega$ grows at most linearly, 
\be\label{pointwiseEstimateXkay}
|X^k(y,t)|\le B(1+t)(1+|y|)
\ee
for some constant $B>0$ independent of $k\in \N$ and for each $y\in \R$ and $t\ge 0$. 

\par It follows from the Arzel\`a-Ascoli theorem that $(X^{k_j}(t))_{j\in \N}$ has a subsequence 
$(X^{k_{j_\ell}}(t))_{\ell\in \N}$ that converges locally uniformly on $\R$ to a uniformly continuous function $Y:\R\rightarrow \R$.
We also have by \eqref{weakLimXkayJay} that
$$
\int_{\R} Y\;\phi d\rho_0=\lim_{\ell\rightarrow\infty}\int_{\R}X^{k_{j_\ell}}(t)\;\phi d\rho^{k_{j_\ell}}_0=\int_{\R} X(t)\;\phi d\rho_0
$$
for $\phi\in C_b(\R)$. That is, $X(t)=Y$ $\rho_0$ almost everywhere.  And for any another subsequence of $(X^{k_j}(t))_{j\in \N}$ which 
converges locally uniformly to a continuous function $Z$, it must be that $Y=Z$ $\rho_0$ almost everywhere. 

\par If $Y(x_0)>Z(x_0)$ for 
some $x_0\in\text{supp}(\rho_0)$, then continuity ensures $Y>Z$ in some neighborhood $(x_0-\delta,x_0+\delta)$ of $x_0$. 
This leads to a contradiction 
$$
0=\int_{\R}|Y-Z|d\rho_0\ge \int_{(x_0-\delta,x_0+\delta)}(Y-Z)d\rho_0>0,
$$
since $\rho_0((x_0-\delta,x_0+\delta))>0$. It follows that $Z=Y$ on the support of $\rho_0$, and these limiting values are uniquely determined on the support of $\rho_0$. 

\par  Without any loss of generality, we will redefine $X(t)=Y$ as these functions agree $\rho_0$ almost everywhere and now note
\be\label{XkayjayConvUnifYes}
\begin{cases}
X^{k_j}(y^j,t)\rightarrow X(y,t)\\
\text{whenever $y\in\text{supp}(\rho_0)$ and $y^j\rightarrow y$}.
\end{cases}
\ee
Moreover, in view of the bound \eqref{pointwiseEstimateXkay}, 
we can also apply Lemma \ref{LemmaVariantNarrowCon} to get 
$$
\lim_{\ell\rightarrow\infty}\int_{\R}X^{k_{j_\ell}}(t)^2d\rho^{k_{j_\ell}}_0=\int_{\R}Y^2d\rho_0=\int_{\R}X(t)^2d\rho_0<\infty.
$$
As this limit is independent of the subsequence, we actually have 
$$
\lim_{j\rightarrow\infty}\int_{\R}X^{k_{j}}(t)^2d\rho^{k_{j}}_0=\int_{\R}X(t)^2d\rho_0.
$$
The limit \eqref{strongConvXkayJay} now follows as we have shown that $(x,y)\mapsto x^2+y^2$ is uniformly integrable with respect to $\sigma^{k_j}_t$ (see Remark 7.1.1 of \cite{AGS} for more on this technical point).

\par  \underline{Step 4: verifying $(i)$, $(ii)$ and $(iii)$}.  Let us now define the mapping $X:[0,\infty)\rightarrow L^2(\rho_0); t\mapsto X(t)$ and let $0\le s\le t$. By \eqref{XLipschitzEst} and the assumption that $v_0:\R\rightarrow \R$ is absolutely continuous and grows at most linearly,
\begin{align*}
\int_{\R}(X(t)-X(s))^2d\rho_0
&=\lim_{j\rightarrow\infty}\int_{\R}(X^{k_j}(t)-X^{k_j}(s))^2d\rho^{k_j}_0\\
&\le  (t-s)^2\lim_{j\rightarrow\infty}\int_{\R}v_0^2d\rho^{k_j}_0\\
&=  (t-s)^2\int_{\R}v_0^2d\rho_0.
\end{align*}
It follows that $X$ is Lipschitz continuous. 

%XkayjayConvUnifYes
\par Suppose $y,z\in\text{supp}(\rho_0)$ with $y<z$.  By Proposition 5.1.8 of \cite{AGS}, there are sequences $(y^j)_{j\in \N}$ and $(z^j)_{j\in \N}$ with $y^j,z^j\in\text{supp}(\rho_0^{k_j})$ such that $y^j\rightarrow y$ and $z^j\rightarrow z$. Without any loss of generality, we may suppose that $y^j<z^j$ for all $j\in \N$.  By part $(iv)$ of Proposition \ref{DiscreteSolnProp}, 
\be
0\le X^{k_j}(z^j,t)-X^{k_j}(y^j,t)\le z^j-y^j+t\int^{z^j}_{y^j}|v'_0(x)|dx
\ee
for $j\in \N$.  In view of \eqref{XkayjayConvUnifYes}, we can send $j\rightarrow\infty$ and conclude part $(i)$ of this theorem.  A similar argument combined with part $(v)$ of Proposition \ref{DiscreteSolnProp} can be used to prove part $(ii)$ of this theorem.  We leave the details to the reader. 

\par Let us finally verify part $(iii)$ of this theorem. To this end, we $0<s\le t$ and recall from part $(vi)$ of Proposition \ref{DiscreteSolnProp} that there is $f^{k_j}: \R\rightarrow \R$ which satisfies 
$$
|f^{k_j}_{t,s}(x)-f^{k_j}_{t,s}(y)|\le \frac{t}{s}|x-y|, \quad x,y\in \R
$$
and 
\be\label{MeasureabilityforlimitXkayjay}
X^{k_j}(y^j,t)=f^{k_j}_{t,s}(X^{k_j}(y^j,s))
\ee
for $y^j$ belonging to the support of $\rho^{k_j}_0$. Choose $y\in \text{supp}(\rho_0)$ and $y^j\rightarrow y$. By \eqref{XkayjayConvUnifYes}, $X^{k_j}(y^j,s)\rightarrow X(y,s)$ and 
$X^{k_j}(y^j,t)\rightarrow X(y,t)$. As
\begin{align*}
|f^{k_j}_{t,s}(x)|&\le |f^{k_j}_{t,s}(x)-f^{k_j}_{t,s}(X^{k_j}(y^j,s))|+|f^{k_j}_{t,s}(X^{k_j}(y^j,s))|\\
&\le \frac{t}{s}|x-X^{k_j}(y^j,s)|+|X^{k_j}(y^j,t)|,
\end{align*}
$f^{k_j}$ is locally uniformly bounded on $\R$. It follows that $f^{k_j}$ has a subsequence (which we will not relabel) which converges locally uniformly on $\R$ to a function $f$ which satisfies the same Lipschitz estimate. Sending $k_j\rightarrow\infty$ along an appropriate sequence in \eqref{MeasureabilityforlimitXkayjay} gives \eqref{MeasureabilityforlimitX}. 
\end{proof}
For the remainder of this subsection, we will denote $X$ as the mapping and $(X^{k_j})_{j\in \N}$ as the sequence obtained in the previous lemma. We note that as $X: [0,\infty)\rightarrow L^2(\rho_0)$ 
is Lipschitz continuous it is differentiable almost everywhere on $[0,\infty)$. 
\begin{cor}\label{dotufunctionXtee}
For almost every $t\ge0$, there is a Borel function $u: \R\rightarrow \R$ such that  
$$
\dot X(t)=u(X(t))
$$
$\rho_0$ almost everywhere. 
\end{cor}
\begin{proof}
Choose a time $t\ge 0$ for which 
$$
\dot X(t)=\lim_{n\rightarrow\infty}n\left(X(t+1/n)-X(t)\right)
$$
exists in $L^2(\rho_0)$. Without any loss of generality, we may assume this limit exists $\rho_0$ almost everywhere as it does for a subsequence.  
By part $(iii)$ of Lemma \ref{FirstCompactnessXkayJay}, 
\be\label{dotXexists}
\dot X(t)=\lim_{n\rightarrow\infty}u_n(X(t))
\ee
$\rho_0$ almost everywhere. Here
$$
u_n:=n\left(f_{t+1/n,t}-\id\right)
$$
is Borel measurable for each $n\in \N$. 

\par Let $S\subset \R$ be a Borel subset such that $\rho_0(S)=1$ and \eqref{dotXexists} holds at each point in $S$; such a subset can be found as detailed in Theorem 1.19 in \cite{Folland}. Let us also define 
the Borel sigma sub-algebra
$$
{\cal F}:=\left\{\{y\in S: X(y,t)\in A\} :  A\subset \R\; \text{Borel}\right\}.
$$
We note that ${\cal F}$ is the sigma algebra generated by the restriction of $X(t)$ to $S$, so a Borel function is ${\cal F}$ measurable if and only if it is a composition of a Borel function with $X(t)|_{S}$ (exercise 1.3.8 of \cite{MR2722836}).  Consequently, $\dot X(t)|_S$ is the pointwise limit of ${\cal F}$ measurable functions and therefore must be ${\cal F}$ measurable itself (Corollary 2.9 \cite{Folland}). As a result, there is some Borel $u: \R\rightarrow \R$ for which 
$$
\dot X(t)|_S=u\left( X(t)|_S\right).
$$
That is, $\dot X(t)=u(X(t))$ $\rho_0$ almost everywhere. 
\end{proof}
The final lemma needed for the proof of Theorem \ref{mainThm} is as follows. 
\begin{lem}\label{LemCondotXkaryJay}
Suppose $t\ge 0$ and $g\in C_b(\R)$.  Then
\be\label{WeakConvdotX1}
\lim_{j\rightarrow\infty}\int^t_0\int_{\R}\dot X^{k_j}(\tau)\;g(X^{k_j}(\tau)) d\rho^{k_j}_0d\tau=
\int^t_0\int_{\R}\dot X(\tau)\;g(X(\tau)) d\rho_0d\tau.
\ee
\end{lem}

\begin{proof}
Set 
\be\label{antiderivativeGee}
F(z)=\int^z_0 g(y)dy
\ee
for $z\in \R$ and observe that $F$ is continuously differentiable and Lipschitz continuous.  Moreover, 
\begin{align*}
\int^t_0\int_{\R}\dot X^{k_j}(\tau)\;g(X^{k_j}(\tau))  d\rho^{k_j}_0d\tau&=\int_{\R}\left(\int^t_0\dot X^{k_j}(\tau)\;g(X^{k_j}(\tau)) d\tau \right)  d\rho^{k_j}_0\\
&=\int_{\R}\left(\int^t_0\frac{d}{d\tau}F(X^{k_j}(\tau)) d\tau \right)  d\rho^{k_j}_0\\
&=\int_{\R}\left(F(X^{k_j}(t))-F(\id) \right) d\rho^{k_j}_0.
\end{align*}
Since $F$ grows at most linearly, we can appeal to Lemma \ref{FirstCompactnessXkayJay} and send $j\rightarrow\infty$ to find  
\begin{align*}
\lim_{j\rightarrow\infty}\int^t_0\int_{\R}\dot X^{k_j}(\tau)\;g(X^{k_j}(\tau)) d\rho^{k_j}_0d\tau&=\int_{\R}\left(F(X(t))-F(\id) \right) d\rho_0\\
&=\int_{\R}\left(\int^t_0\frac{d}{d\tau}F(X(\tau))\right) d\rho_0\\
&=\int_{\R}\left(\int^t_0\dot X(\tau)\;g(X(\tau)) d\tau \right) d\rho_0\\
&=\int^t_0\int_{\R}\dot X(\tau)\;g(X(\tau)) d\rho_0d\tau.
\end{align*}
\end{proof}
\begin{proof}[Proof of Theorem \ref{mainThm}]
We will show $X:[0,\infty)\rightarrow L^2(\rho_0)$ is the desired solution. First note that Lemma \ref{FirstCompactnessXkayJay} implies 
$$
\int_{\R}X(0)\phi d\rho_0=\lim_{j\rightarrow\infty}\int_{\R}X^{k_j}(0)\phi d\rho^{k_j}_0=\lim_{j\rightarrow\infty}\int_{\R}\id\phi d\rho^{k_j}_0
=\int_{\R}\id\phi d\rho_0
$$
for each $\phi\in C_b(\R)$. It follows that $X$ satisfies the initial condition \eqref{FlowInit}.  It also follows from \eqref{strongConvXkayJay} that
$$
\lim_{j\rightarrow\infty}\int^t_0\int_{\R}v_0 g(X^{k_j}(\tau))d\rho^{k_j}_0d\tau=\int^t_0\int_{\R}v_0 g(X(\tau))d\rho_0d\tau
$$
for each $g\in C_b(\R)$ and $t\ge0$. Combining with Lemma \ref{LemCondotXkaryJay} gives 
\begin{align*}
\int^t_0\int_{\R}v_0 g(X(\tau))d\rho_0d\tau&=\lim_{j\rightarrow\infty}\int^t_0\int_{\R}v_0 g(X^{k_j}(\tau))d\rho^{k_j}_0d\tau\\
&=\lim_{j\rightarrow\infty}\int^t_0\int_{\R}\dot X^{k_j}(\tau)g(X^{k_j}(\tau))d\rho^{k_j}_0d\tau\\
&=\int^t_0\int_{\R}\dot X(\tau)g(X(\tau))d\rho_0d\tau.
\end{align*}

\par We may write
\be\label{IntegratedCondExp}
\int^t_0\int_{\R}\dot X(\tau)g(X(\tau))d\rho_0d\tau=\int_{\R}\left(F(X(t))-F(\id)\right)d\rho_0
\ee
using an antiderivative $F$ of $g$ as in \eqref{antiderivativeGee}.  Recall that $\dot X(t)$ exists for almost every $t\ge 0$. At any such $t$, we can differentiate \eqref{IntegratedCondExp} to find
$$
\int_{\R}\dot X(t)g(X(t))d\rho_0d\tau=\int_{\R}v_0 g(X(t))d\rho_0.
$$
By Corollary \eqref{dotufunctionXtee}, there is also a Borel function $u:\R\rightarrow \R$ such that
$$
\dot X(t)=u(X(t))
$$
for almost every $t\ge 0$. These observations imply that $X$ satisfies the sticky particle flow equation \eqref{FlowEqn}.

\par Part $(ii)$ and $(iii)$ of this theorem follows from parts $(i)$ and $(ii)$ of Lemma \ref{FirstCompactnessXkayJay}, respectively.    So all that we are left to show is part $(i)$. Fix two times $s,t\ge 0$ with $t\ge s$ such that 
$$
\dot X(s)=\E_{\rho_0}[v_0|X(s)]\quad \text{and}\quad \dot X(t)=\E_{\rho_0}[v_0|X(t)]
$$
$\rho_0$ almost everywhere. By part $(iii)$ of Lemma \ref{FirstCompactnessXkayJay} and the tower property of conditional expectation,
\eqref{TowerProp}
$$
\E_{\rho_0}[\dot X(s)|X(t)]= \E_{\rho_0}[\E_{\rho_0}[v_0|X(s)]|X(t)]=\E_{\rho_0}[v_0|X(t)]=\dot X(t).
$$
We then conclude
$$
\int_{\R}\frac{1}{2}\dot X(t)^2d\rho_0\le \int_{\R}\frac{1}{2}\dot X(s)^2d\rho_0\le \int_{\R}\frac{1}{2}v_0^2d\rho_0
$$
by appealing to \eqref{JensenINeq}. 
\end{proof}

%---------------------------------------------------------------------
\subsection{Generating a solution of the SPS}
This final subsection is dedicated to the Proof of Corollary \ref{oldThm}, which we will accomplish in three steps. 
\par 1. For each $t\ge 0$, set 
$$
\rho_t:=X(t)_{\#}\rho_0.
$$ 
As $X:[0,\infty)\rightarrow L^2(\rho_0)$ is continuous, $\rho: [0,\infty)\rightarrow {\cal P}(\R)$ is narrowly continuous.  Let us also define the Borel probability measure $\mu$ on $\R\times [0,\infty)$ 
$$
\mu(S):=\int^\infty_0\int_{\R} \chi_S(X(t),t)d\rho_0e^{-t}dt=\iint_Sd\rho_te^{-t}dt
$$
and the signed Borel measure $\pi$ on $\R\times [0,\infty)$  
$$
\pi(S):=\int^\infty_0\int_{\R}v_0\cdot \chi_S(X(t),t)d\rho_0e^{-t}dt
$$
for $S\subset \R\times [0,\infty)$. 

\par In view of H\"older's inequality, 
$$
|\pi(S)|\le \left(\int_{\R}|v_0|^2d\rho_0\right)^{1/2}\mu(S)^{1/2}.
$$
Therefore, $\pi$ is absolutely continuous with respect to $\mu$.  By the Radon-Nikodym theorem, there is a Borel $v:\R\times[0,\infty)\rightarrow \R$ such that 
\begin{align*}
\int^\infty_0\int_{\R}h(x,t)d\pi(x,t)&=\int^\infty_0\int_{\R}h(X(t),t)v_0d\rho_0e^{-t}dt\\
&=\int^\infty_0\int_{\R}h(x,t)v(x,t)d\mu(x,t)\\
&=\int^\infty_0\int_{\R}h(x,t)v(x,t)d\rho_t(x)e^{-t}dt\\
&=\int^\infty_0\int_{\R}h(X(t),t)v(X(t),t)d\rho_0e^{-t}dt.
\end{align*}
\par It follows that for Lebesgue almost every $t\ge 0$,
\be\label{visCondExp}
v(X(t),t)=\E_{\rho_0}[v_0|X(t)]=\dot X(t)
\ee
$\rho_0$ almost everywhere.  Also note
$$
\int_{\R}v(x,t)^2d\rho_t(x)=\int_{\R}v(X(t),t)^2d\rho_0=\int_{\R}\dot X(t)^2d\rho_0\le \int_{\R}v_0^2d\rho_0,
$$
for almost every $t\ge 0$. Therefore
$$
\int^T_0\int_{\R}v(x,t)^2d\rho_t(x)dt<\infty
$$
for each $T>0$. 

\par 2. Fix $\phi\in C^\infty_c(\R\times[0,\infty))$ and observe
\begin{align}
\int^\infty_0\int_{\R}(\partial_t\phi+v\partial_x\phi)d\rho_tdt
&=\int^\infty_0\int_{\R}(\partial_t\phi(X(t),t)+v(X(t),t)\partial_x\phi(X(t),t))d\rho_0dt\\
&=\int^\infty_0\int_{\R}(\partial_t\phi(X(t),t)+\dot X(t)\partial_x\phi(X(t),t))d\rho_0dt\\
&=\int^\infty_0\int_{\R}\frac{d}{dt}\phi(X(t),t)d\rho_0dt\\
&=\int_{\R}\int^\infty_0\frac{d}{dt}\phi(X(t),t)dt d\rho_0\\
&=-\int_{\R}\phi(X(0),0)d\rho_0\\
&=-\int_{\R}\phi(\cdot,0)d\rho_0.
\end{align}
We also have by \eqref{visCondExp},
\begin{align}
&\int^\infty_0\int_{\R}(v\partial_t\phi+v^2\partial_x\phi)d\rho_tdt\\
&\quad=\int^\infty_0\int_{\R}(\partial_t\phi(X(t),t)+v(X(t),t)\partial_x\phi(X(t),t))v(X(t),t)d\rho_0dt\\
&\quad=\int^\infty_0\int_{\R}(\partial_t\phi(X(t),t)+v(X(t),t)\partial_x\phi(X(t),t))v_0d\rho_0dt\\
&\quad=\int^\infty_0\int_{\R}(\partial_t\phi(X(t),t)+\dot X(t)\partial_x\phi(X(t),t))v_0d\rho_0dt\\
&\quad=\int^\infty_0\int_{\R}\frac{d}{dt}\phi(X(t),t)v_0d\rho_0dt\\
&\quad=\int_{\R}\left(\int^\infty_0\frac{d}{dt}\phi(X(t),t) dt\right)v_0d\rho_0\\
&\quad=-\int_{\R}\phi(X(0),0)v_0d\rho_0\\
&\quad=-\int_{\R}\phi(\cdot,0)v_0d\rho_0.
\end{align}
As a result, the pair $\rho$ and $v$ is a weak solution of the SPS \eqref{SPS} with initial conditions \eqref{Init}. 

\par 3. In view of \eqref{visCondExp} and $(i)$ of Theorem \ref{mainThm}  
\begin{align*}
\int_{\R}\frac{1}{2}v(x,t)^2d\rho_t(x)&=\int_{\R}\frac{1}{2}\dot X(t)^2d\rho_0\\
&\le \int_{\R}\frac{1}{2}\dot X(s)^2d\rho_0\\
&= \int_{\R}\frac{1}{2}v(x,s)^2d\rho_s(x)
\end{align*}
for almost every $0\le s\le t$.   Moreover, part $(iii)$ of Theorem \ref{mainThm} implies
\begin{align*}
0&\ge \frac{d}{dt}\frac{1}{t^2}(X(w,t)-X(z,t))^2\\
&=\frac{2}{t^2}(X(w,t)-X(z,t))(\partial_tX(w,t)-\partial_tX(z,t))-\frac{2}{t^3}(X(w,t)-X(z,t))^2\\
&=\frac{2}{t^2}\left[(X(w,t)-X(z,t))(v(X(w,t),t)-v(X(z,t),t))-\frac{1}{t}(X(w,t)-X(z,t))^2\right]
\end{align*}
for Lebesgue almost every $t>0$ and $w,z\in E$. Here $E\subset \R$ is $\rho_0$ measurable and $\rho_0(E)=1$. Without loss of generality, we may assume $E$ is a countable union of closed sets 
(part $c$ of Theorem 1.19 in \cite{Folland}). 

\par In particular, we have shown that \eqref{Entropy} holds for $x,y$ belonging to the forward image of $E$ under $X(t)$
$$
X(t)(E):=\{X(w,t)\in \R: w\in E\}.
$$
By part $(ii)$ of Theorem \ref{mainThm}, we may assume that $X(t): \R\rightarrow \R$ is continuous. It follows that $X(t)(E)$ is Borel measurable (see Proposition \ref{measProp}). Furthermore, 
$$
X(t)^{-1}\left[X(t)(E)\right]\supset E,
$$
so
$$
\rho_t(X(t)(E))=\rho_0(X(t)^{-1}\left[X(t)(E)\right])
\ge \rho_0(E)=1.
$$
Consequently,  \eqref{Entropy} holds on a Borel subset of full measure for $\rho_t$ and
we conclude part $(ii)$ of this corollary.

\appendix

\section{Measurability of a continuous image}
In this appendix, we will prove the following elementary assertion which was used in the proof of Corollary \ref{oldThm}.  
\begin{prop}\label{measProp}
Suppose $f: \R\rightarrow \R$ is continuous and $C=\bigcup_{i\in \N}C_i$ and each $C_i\subset\R$ is closed. Then $f(C)$ is Borel measurable. 
\end{prop}
\begin{proof}
For each $i\in \N$, we may write 
$$
C_i=\R\cap C_i=\left(\bigcup_{k\in \Z}[k,k+1]\right)\cap C_i=\bigcup_{k\in \Z}\left([k,k+1]\cap C_i\right).
$$
As the forward image distributes over unions,
$$
f(C_i)=\bigcup_{k\in \Z}f([k,k+1]\cap C_i).
$$
Since $[k,k+1]\cap C_i$ is compact and $f$ is continuous, $f([k,k+1]\cap C_i)$ is compact. As a result, $f(C_i)$ is a countable union of compact subsets of $\R$ and is thus Borel measurable. Hence, 
$$
f(C)=\bigcup_{i\in \N}f(C_i)
$$
is also Borel.
\end{proof}

% References
\bibliography{SPbib}{}

\bibliographystyle{plain}

\typeout{get arXiv to do 4 passes: Label(s) may have changed. Rerun}

\end{document}